\documentclass[12pt]{article}
\usepackage[utf8]{inputenc}
\usepackage{amssymb,amsmath,amsfonts,amsthm,amscd,latexsym,indentfirst,verbatim,xcolor}
\usepackage[T2A]{fontenc}
\usepackage{geometry}
\def\dbl{\lbrace\kern-3pt\lbrace}
\def\dbr{\rbrace\kern-3pt\rbrace}

\def\charr{\textrm{char}\,}
\def\Imm{\textrm{Im}\,}

\def\Aut{\textrm{Aut}}

\def\id{\operatorname{id}}

\newcommand{\lc}{\operatorname{lc}}

\theoremstyle{plain}
\newtheorem{theorem}{Theorem}[section]
\newtheorem{lemma}[theorem]{Lemma}

\newtheorem*{conjecture*}{Conjecture}

\newtheorem{proposition}[theorem]{Proposition}

\theoremstyle{definition}
\newtheorem{definition}[theorem]{Definition}

\begin{document}
\sloppy
\hfill{16W99 (MSC2020)}

\begin{center}
{\Large
Injective Rota--Baxter operators of weight~1 on $F[x]$}

\smallskip

Vsevolod Gubarev
\end{center}

\begin{abstract}
We describe all injective Rota--Baxter operators~$R$ of weight~1 on the polynomial algebra
$F[x]$. 
When $\charr F = p>0$, the only one is $R=-\id$. 
When $\charr F = 0$, we have either $R = -\id$ or, up to conjugation with automorphisms of $F[x]$, $R(1) = x$, and $R$ is uniquely defined via this equality.
Together with the known weight-zero case, this completes the classification
of injective Rota--Baxter operators of any weight on $F[x]$.

\end{abstract}

\section{Introduction}

Rota--Baxter operators on algebras have been studied since 1960, when G. Baxter~\cite{Baxter} gave the general definition to provide a rigorous algebraic foundation for
Spitzer's identity from fluctuation theory.

Since then, numerous connections and applications have been discovered; see, for example,~\cite{GuoMonograph}.
Most classification results of Rota--Baxter operators are devoted to finite-dimensional algebras.
The study of Rota--Baxter operators on the polynomial algebra $F[x]$ began relatively recently.

In 2015, S.~H.~Zheng, L.~Guo, and M.~Rosenkranz described~\cite{Monom2} all injective monomial Rota--Baxter operators of weight zero on $F[x]$ over a field of characteristic zero; a~monomial operator is an operator which maps each monomial to a monomial with some coefficient. They proved
that up to a~constant term every monomial injective Rota--Baxter operator on $F[x]$ acts on each monomial as a~composition of the multiplication $l_r$ by a fixed nonzero polynomial $r$
and a~formal integration $J_a$ at some fixed point~$a$. Moreover, S.~H.~Zheng, L.~Guo, and M.~Rosenkranz formulated a conjecture that this result holds true for any injective Rota--Baxter operator on $\mathbb{R}[x]$.

In 2021, the author and A. Perepechko confirmed the conjecture of 
S.~H.~Zheng, L.~Guo, and M.~Rosenkranz over any field of characteristic zero~\cite{InjectiveZero}.
Since there are no nonzero Rota--Baxter operators of weight~0 on $F[x]$, when $\charr F>0$~\cite{Khodzitskii3}, this finished the description of injective Rota--Baxter operators of weight~0 on $F[x]$ over any field~$F$.

The current work is devoted to the case of nonzero weight.
This case is quite different, since now we deal with two Rota--Baxter operators of the same weight~$\lambda\neq0$ on~$F[x]$: $R$ and $R'=-R-\lambda \id$.
We prove (over any field) that if only one of them is injective, then such operator is trivial, i.\,e. $R = -\id$.
The remaining case, when both $R$ and $R'$ are injective, can occur only in characteristic zero, and this case gives, up to conjugation with automorphisms of $F[x]$, only one additional operator $R$ such that $R(1) = x$ and the action $R(x^k)$ is defined via Faulhaber polynomials.

Therefore we obtain the complete description of injective Rota--Baxter operators on $F[x]$.
The core part of the proof, Proposition~\ref{prop:Main}, was obtained by Aristotle AI~\cite{Aristotle}.

Monomial Rota--Baxter operators on $F[x]$ were studied  in~\cite{MonomNonunital,Monom}.
In a series of works, A. Khodzitskii studied monomial Rota---Baxter operators on the polynomial algebra $F[x,y]$~\cite{Khodzitskii,Khodzitskii2}.

\section{Preliminaries}

\begin{definition}
A linear operator $R$ defined on an algebra $A$
is called a Rota--Baxter operator (RB-operator, for short) of weight $\lambda$, where $\lambda\in F$, if the identity
\begin{equation}\label{RB}
R(f)R(g) = R( R(f)g + fR(g) )+\lambda R(fg)
\end{equation}
holds for every $f,g\in A$.
\end{definition}

\begin{lemma}\label{lemma:phi}
Given an RB-operator $R$ of weight $\lambda$,

(a) the operator $R':=-R-\lambda\id$ is an RB-operator of weight $\lambda$,

(b) the operator $\lambda^{-1}R$ is an RB-operator of weight 1, provided $\lambda\neq0$.
\end{lemma}

\begin{lemma}\label{lem:Aut}
Given an algebra $A$, an RB-operator $P$ on $A$ of weight $\lambda$,
and $\psi\in\Aut(A)$, the operator $P^{(\psi)} = \psi^{-1}P\psi$
is an RB-operator of weight~$\lambda$ on~$A$.
\end{lemma}

Let $F_n(m) = \sum\limits_{j=1}^{m}j^n$ for natural $n,m$.
It is well known that $F_1(m) = m(m+1)/2$,
$F_2(m) = m(m+1)(2m+1)/6$, $F_3(m) = (F_1(m))^2$.
For any $n$,
\begin{equation}\label{sum-formula}
F_n(m) = \frac{1}{n+1}\sum\limits_{j=0}^n(-1)^j
 \binom{n+1}{j} B_j m^{n+1-j},
\end{equation}
where $B_0=1,B_1,\ldots,B_n$ are Bernoulli numbers.

\begin{lemma}[\cite{Unital,Miller66}]\label{lem:Faulhaber}
Let $A$ be a unital associative algebra over a field~$F$ of characteristic zero, $R$~be an RB-operator
of weight~1 on~$A$, $a = R(1)$.
Then $R(a^n) = (-1)^{n+1}F_n(-a)$ for all $n\in\mathbb{N}_{>0}$.
\end{lemma}

Define an RB-operator~$R_0$ of weight~1 on $F[x]$, when $\charr F = 0$, exactly by the formula from~Lemma~\ref{lem:Faulhaber}:
$R_0(1) = x$, $R_0(x^n) = (-1)^{n+1}F_n(-x)$, it gives exactly the free unital commutative Rota--Baxter algebra of weight~1.
Note that $R_0$ as well as $R_0'$ are injective.

\section{Case when $R'$ is not injective}

Assume from now on that $R$ is an injective RB-operator of weight~1 on $F[x]$. 

\begin{proposition}\label{coro:R'IsNotInjective}
Let $F$ be any field. If $R\colon F[x]\to F[x]$ is injective RB-operator of weight~1, but $R'$ is not injective, then $R=-\id$.
\end{proposition}

\begin{proof}
We apply the well-known facts that $\ker R$ is an ideal in $\Imm R'$ as well as $\ker R'$ is an ideal in $\Imm R$ and moreover, the following isomorphism holds:
\begin{equation} \label{iso}
\Imm R/\ker(R') \cong \Imm R'/\ker (R).
\end{equation}

If $R'$ is not injective, then the left-hand side of~\eqref{iso}
is finite-dimensional; hence $\Imm R'$ is finite-dimensional. Thus, $\Imm R' = F$ or $\Imm R' = (0)$.
The first case is impossible: writing
$R' = \varphi$,~\eqref{RB} gives
$\varphi(fg)+\varphi(f)\varphi(g)=0$.
Taking $g=1$ and using $\varphi\neq0$, we get $\varphi(1)=-1$. Hence $R(1)=0$, a~contradiction.
Therefore $\Imm R' = 0$ and $R = -\id$.
\end{proof}

\section{Case when $R'$ is also injective}

Assume now that both $R$ and $R'$ are injective. 
Define a~linear map $\sigma \colon \Imm R\to F[x]$ by
$\sigma(R(f))=R(f)+f$. 
This is well defined since $R$ is injective. 

\begin{lemma}\label{lem:injectivitySigma}
The map $\sigma$ is a homomorphism from $\Imm R$ to $F[x]$ and $\sigma$ is injective exactly when $R'$ is injective.
\end{lemma}

\begin{proof}
Set $w=R(f)g+fR(g)+fg$. By~\eqref{RB}, we have $R(f)R(g)=R(w)$, and hence
\begin{multline*}
\sigma(R(f)R(g)) 
 =R(w)+w =R(f)R(g)+R(f)g+fR(g)+fg \\
 =(R(f)+f)(R(g)+g)
 = \sigma(R(f))\sigma(R(g)).
\end{multline*}
If $R'(f)=0$ for some $f\neq0$, then $R(f)\neq0$ and $\sigma(R(f))=0$.
\end{proof}

Since $R'$ is injective, Lemma~\ref{lem:injectivitySigma} shows that $\sigma$ also is injective.

Put $D=\sigma-\id$. Note that $D$ is a linear bijection and $D^{-1} = R$.

\begin{lemma}\label{lem:NoConstant}
$\Imm R$ contains no nonzero constant.
\end{lemma}

\begin{proof}
If $0\neq c\in \Imm R\cap F$, then $1\in \Imm R$. The element $\sigma(1)$ is idempotent since $\sigma$ is a~homomorphism. The domain $F[x]$ has only the idempotents~0 and~1, and injectivity of~$\sigma$ excludes~0.
Thus $\sigma(1)=1$, so $D(1)=0$, contradicting injectivity of~$D$.
\end{proof}

Let $q=D^{-1}(1)$. Thus $q\in \Imm R$ and
$\sigma(q)=q+1$.
By Lemma~\ref{lem:NoConstant}, $q$~is nonconstant.

\begin{lemma}\label{lem:charF}
If both $R$ and $R'$ are injective, then $\charr F = 0$.
\end{lemma}

\begin{proof}
Suppose that $\charr F = p>0$. Since $q,q^p\in \Imm R$, the polynomial $m=q^p-q$ belongs to $\Imm R$. Further,
$$
\sigma(m)=(q+1)^p-(q+1)=q^p-q=m.
$$
Thus $D(m)=0$, and injectivity of $D$ gives $m=0$. This is impossible for nonconstant $q$, since
$\deg(q^p)=p\deg q>\deg q$.
Hence $\charr F = 0$.
\end{proof}

\begin{proposition} \label{prop:Main}
Let $\charr F = 0$.
Let $A\subseteq F[x]$ be a nonzero subalgebra, and let $\sigma\colon A\to F[x]$ be an injective linear homomorphism. Suppose
$D=\sigma-\id\colon A\to F[x]$ is a linear bijection.
Then there are unique $a\in F$ and $h\in F^\times$ such that
$A=(x-a)F[x]$ and $\sigma(m)(x)=m(x+h)$ for every $m\in A$.
\end{proposition}

\begin{proof}
By Lemma~\ref{lem:NoConstant}, $A$ contains no nonzero constant. Let $q=D^{-1}(1)$. Then $q$ is nonconstant and
$\sigma(q)=q+1$. Put
$$
K=F(q),\quad L=F(x),\quad B=F[q].
$$
The extension $L/K$ is finite: if $q=q(x)$ has degree $d\geq1$, then $x$ is a root of $q(T)-q\in K[T]$.

The algebra $A$ is a $B$-module.
Let $A_K$ be the $K$-linear span of $A$ inside $L$. It is a~finite-dimensional $K$-subalgebra of the field $L$. It contains~1, since $q\in A$ and $1=q^{-1}q$. It is known that a~finite-dimensional domain over a field is a field. Hence
$K\subseteq A_K\subseteq L$ is a tower of fields.

Let $\tau \colon K\to K$ be the automorphism determined by $\tau(q)=q+1$. Semilinearity extends $\sigma$ to a $\tau$-semilinear field embedding $\sigma_K \colon A_K\to L$.
Explicitly,
$$
\sigma_K\Big(\sum_i k_i a_i\Big)=\sum_i \tau(k_i)\sigma(a_i).
$$
This is well defined: after clearing denominators, a $K$-linear relation among the $a_i$ becomes a relation $\sum_i g_i(q)a_i=0$ inside $A$; applying $\sigma$ and using semilinearity gives the required twisted relation. The image $\sigma_K(A_K)$ is another intermediate field, and
$[\sigma_K(A_K):K]=[A_K:K]$.

Surjectivity of $D$ gives
$A+\sigma(A)=F[x]$, since every $f$ has the form $f=\sigma(a)-a$. Taking $K$-linear spans inside $L$ yields
$A_K+\sigma_K(A_K)=L$.
Write $n=[L:K]$, $r=[A_K:K]$.
The last equality gives $n\leq2r$, while the tower law gives $n=[L:A_K]r$. Hence $[L:A_K]\leq2$. If $[L:A_K]=2$, then $n=2r$ and the dimension inequality is an equality. It would follow that $A_K\cap\sigma_K(A_K)=0$. This is impossible because both fields contain~$K$. Therefore
$A_K=L$.

Thus $\sigma_K$ is an $F$-algebra endomorphism of $F(x)$ satisfying $\sigma_K(q)=q+1$. Put $w=\sigma_K(x)\in F(x)$. Since an algebra endomorphism acts on rational functions by substitution, $q(w)=q+1$.

Write $w=P/Q$ with coprime $P,Q\in F[x]$. If $Q$ were nonconstant, then the denominator $Q^{\deg q}$ in $q(P/Q)$ could not cancel: modulo $Q$, the numerator is the leading coefficient of $q$ times $P^{\deg q}$, which is coprime to $Q$. Since $q(w)=q+1$ is a polynomial, $Q$ must be constant. Hence $w\in F[x]$.

Degree comparison gives
$$
\deg q \cdot \deg w
 = \deg(q(w))
 = \deg(q+1)
 = \deg q,
$$
so $\deg w=1$. Write $w=\alpha x+\beta$ with $\alpha\neq0$. If $\alpha\neq1$, then $w$ has the fixed point $x_0=\beta/(1-\alpha)$, and evaluation of $q(w)=q+1$ at $x_0$ gives the contradiction $q(x_0)=q(x_0)+1$.
Therefore $\alpha=1$. Also $\beta\neq0$. We now have
$q(x+\beta)-q(x)=1$.

If $\deg q=d\geq1$, the left-hand side has degree $d-1$ and leading coefficient $d\beta\,\lc(q)$; this is nonzero in characteristic zero. Since the right-hand side has degree zero, $d=1$.

Write $q=b(x-a)$ with $b\neq0$, and put $h=b^{-1}$.
By linearity of $q$, we have $h = \beta$.
Then $q(x+h)=q(x)+1$.

For every $n\geq0$, the element $q^{n+1}$ belongs to $A$ and
$$
D(q^{n+1})=(q+1)^{n+1}-q^{n+1}.
$$
As a polynomial in $q$, the right-hand side has degree $n$ and leading coefficient $n+1$. Characteristic zero therefore implies, by triangular induction, that
$D\colon qF[q]\longrightarrow F[q]$ is bijective.
Since $q$ is affine, $F[q]=F[x]$. Both this restriction and $D\colon A\to F[x]$ are bijective. If $m\in A$, choose $r\in qF[q] \subseteq A$ with $D(r)=D(m)$; injectivity gives $m=r$. Hence
$A=qF[q]=(x-a)F[x]$.

On $qF[q]$, multiplicativity gives $\sigma(g(q))=g(q+1)$ for every polynomial $g$ with zero constant term. Since translation by $h$ also sends $q$ to $q+1$, we conclude that
$\sigma(m)(x)=m(x+h)$ on all of~$A$.

Finally, the ideal $(x-a)F[x]$ determines $a$: every member vanishes at $a$, and the generator $x-a$ has no other root. Once $a$ is known, applying the translation formula to $x-a$ determines $h$. Thus the pair $(a,h)$ is unique.
\end{proof}

Recall that every automorphism of $F[x]$ is affine:
$x\to \alpha x+\beta$, where $\alpha\in F^\times$, $\beta \in F$.

\begin{theorem}
Let $F$ be a field and let $R\colon F[x]\to F[x]$ be an injective Rota--Baxter operator of weight~1.

a) If $\charr F = p>0$, then $R=-\id$.

b) If $\charr F = 0$, then either $R=-\id$ or, up to conjugation with automorphisms of~$F[x]$, $R=R_0$.
\end{theorem}

\begin{proof}
If $R'$ is not injective, Proposition~\ref{coro:R'IsNotInjective} gives $R=-\id$ over every field.

If $R'$ is injective, Lemma~\ref{lem:charF} shows that $F$ has characteristic zero.

Now, we apply Proposition~\ref{prop:Main}
for $A = \Imm R$ and the map $\sigma$.
Thus, we obtain unique $a\in F$ and $h\neq0$ such that
$\sigma(x-a) = x-a + h$, hence,
$R(1) = (x-a)/h$.
Using conjugation with $\varphi(x) = (x-a)/h$ (see Lemma~\ref{lem:Aut}), we may assume that
$R(1) = x$.
By Lemma~\ref{lem:Faulhaber}, we have
$R(x^n) = (-1)^{n+1}F_n(-x)$ for all $n\geq1$.
Hence, $R = R_0$.
\end{proof}

\section*{Acknowledgements}

The research was carried out within the framework of the Sobolev
Institute of Mathematics state contract (project FWNF-2026-0017).

\noindent Vsevolod Gubarev \\
Sobolev Institute of Mathematics \\
Acad. Koptyug ave. 4, 630090 Novosibirsk, Russia \\
Novosibirsk State University \\
Pirogova str. 1, 630090 Novosibirsk, Russia \\
e-mail: wsewolod89@gmail.com

\end{document}